\documentclass[11pt,a4paper]{article}

\usepackage[a4paper,margin=2.48cm]{geometry}
\usepackage{amsmath,amssymb,amsthm,mathtools}
\usepackage{microtype}
\usepackage{enumitem}
\usepackage{xcolor}
\usepackage[unicode]{hyperref}
\usepackage[nameinlink,noabbrev]{cleveref}

\definecolor{linkblue}{RGB}{34,76,120}
\hypersetup{
  hidelinks,
  pdftitle={Reconstruction of Torsion-Free Abelian Groups from Rational Group Fields},
  pdfauthor={Jinyu Lin, Xiaodong Wang}
}

\setlist{nosep,leftmargin=2.3em}
\numberwithin{equation}{section}
\allowdisplaybreaks

\theoremstyle{plain}
\newtheorem{theorem}{Theorem}[section]
\newtheorem{lemma}[theorem]{Lemma}
\newtheorem{proposition}[theorem]{Proposition}
\newtheorem{corollary}[theorem]{Corollary}
\theoremstyle{definition}
\newtheorem{definition}[theorem]{Definition}
\theoremstyle{remark}
\newtheorem{remark}[theorem]{Remark}

\crefname{theorem}{theorem}{theorems}
\crefname{lemma}{lemma}{lemmas}
\crefname{proposition}{proposition}{propositions}
\crefname{corollary}{corollary}{corollaries}
\crefname{definition}{definition}{definitions}
\crefname{remark}{remark}{remarks}
\crefname{equation}{equation}{equations}

\newcommand{\Q}{\mathbb Q}
\newcommand{\Z}{\mathbb Z}
\newcommand{\N}{\mathbb N}
\newcommand{\Frac}{\operatorname{Frac}}
\newcommand{\trdeg}{\operatorname{trdeg}}
\newcommand{\rk}{\operatorname{rank}}
\newcommand{\Kk}[1]{K_k(#1)}
\newcommand{\KQ}[1]{K_{\Q}(#1)}
\newcommand{\Dk}[1]{\Delta_k(#1)}
\newcommand{\DQ}[1]{\Delta_{\Q}(#1)}
\newcommand{\mon}[1]{X^{#1}}

\title{\bfseries Reconstruction of Torsion-Free Abelian Groups\\
from Rational Group Fields}
\author{%
  Jinyu Lin\textsuperscript{1}, Xiaodong Wang\textsuperscript{2}\\[0.55em]
  \small\textsuperscript{1}\,Taishan College, Shandong University, Jinan, China\\
  \small\textsuperscript{2}\,School of Mathematical Sciences,
  Ocean University of China, Qingdao, China\\[0.25em]
  \small\textsuperscript{1}\,\href{mailto:202400171064@mail.sdu.edu.cn}{202400171064@mail.sdu.edu.cn},
  \textsuperscript{2}\,\href{mailto:wxd8510@stu.ouc.edu.cn}{wxd8510@stu.ouc.edu.cn}%
}
\date{}

\begin{document}

\maketitle

\begin{abstract}
For a torsion-free abelian group $G$, let
\[
  K_{\Q}(G)=\Frac \Q[G]
\]
be the fraction field of its rational group algebra.  We prove that this
field determines the group up to isomorphism:
\[
  K_{\Q}(G)\cong K_{\Q}(H) \quad\Longleftrightarrow\quad G\cong H.
\]
The main structural input is that, over every field $k$ of characteristic
zero, the monomial defect group
\[
  \Delta_k(G)=K_k(G)^\times/(k^\times X^G)
\]
is free abelian.  We give a self-contained proof.  It first treats the
one-variable Puiseux field $F(t^{\Q})$: factorization from
$F(t^{1/n!})$ to $F(t^{1/(n+1)!})$ gives split inclusions because the
substituted irreducibles are square-free in characteristic zero.  A
transfinite decomposition of the divisible hull $\Q\otimes G$ then proves
the general case.

Given a field isomorphism, we compare the two monomial subgroups inside the
common multiplicative group.  Their intersection produces isomorphic
subgroups $M\leq G$ and $N\leq H$, while the quotients $G/M$ and $H/N$
embed in free defect groups and are therefore free.  Relative transcendence
degree shows that these two free quotients have the same rank, completing the
reconstruction.  As consequences, rational-function stabilization of group
fields exactly records free stabilization of groups, and Rickard's bounded
sequence group yields a field $F$ with $F\cong F(x,y)$ but
$F\not\cong F(x)$.
\end{abstract}


\section{Introduction}

The torsion-free abelian groups $G,H,V$, and their subgroups are written
additively.  Unit groups of fields and their quotient groups are written
multiplicatively unless we explicitly pass to additive valuation coordinates.
For a field $k$ and a torsion-free abelian group $G$, the group algebra $k[G]$ has the formal
monomials $\mon g$, $g\in G$, as a $k$-basis, with
\[
  \mon g\mon h=\mon{g+h}.
\]
Since every torsion-free abelian group is orderable, $k[G]$ is an integral
domain.  We put
\[
  \Kk G=\Frac k[G].
\]

Rickard asked whether nonisomorphic torsion-free abelian groups can have
isomorphic rational group fields~\cite{RickardMO2017}.  The question is
particularly natural in infinite rank.  Bastos and Viswanathan proved, among
other closely related statements, that $\KQ G$ is purely transcendental over
$\Q$ precisely when $G$ is free abelian~\cite{BastosViswanathan1988}.  This
does not by itself determine whether arbitrary group fields reflect
isomorphism.

Our main result gives a negative answer to Rickard's question.

\begin{theorem}[Reconstruction theorem]\label{thm:intro-main}
For all torsion-free abelian groups $G$ and $H$,
\[
  \KQ G\cong\KQ H
  \quad\Longleftrightarrow\quad
  G\cong H.
\]
The field isomorphism is not assumed to preserve either group algebra or its
standard monomials.
\end{theorem}

The forward implication is the content of the paper; the reverse implication
is functorial.  No finite-rank, countability, unique-factorization, or
cancellation hypothesis is imposed on $G$ or $H$.

The proof uses the subgroup
\[
  k^\times X^G=\{c\mon g:c\in k^\times,\ g\in G\}
  \leq \Kk G^\times
\]
and the quotient
\begin{equation}\label{eq:defect-intro}
  \Dk G=\Kk G^\times/(k^\times X^G),
\end{equation}
which we call the \emph{monomial defect group}.  The decisive structural
statement is that $\Dk G$ is free abelian in characteristic zero.  The
freeness phenomenon belongs to the classical circle of results initiated by
May~\cite{May1972}; May's later paper explicitly recalls a construction of a
field $L=\Q(\Gamma)$ for which $L^\times/\Gamma$ is free
abelian~\cite[proof of Theorem~1]{May1979}.  We nevertheless give a direct
proof for the canonical group-algebra field.  This both fixes the precise
object under discussion and isolates where characteristic zero is used.

After establishing freeness, a field isomorphism
$\varphi:\KQ G\to\KQ H$ supplies two defect maps, one in each direction.
Their kernels are isomorphic subgroups $M\leq G$ and $N\leq H$.  The
quotients $G/M$ and $H/N$ are free, so both extensions split.  The restriction
of $\varphi$ identifies $\KQ M$ with $\KQ N$, and relative transcendence
degree forces the two complementary free groups to have equal rank.  This
short intersection argument is carried out in \cref{sec:reconstruction}.

\section{Relative Transcendence Degree}

We first prove the rank calculation used in the reconstruction argument.
If $H\leq G$, write
\[
  \rk(G/H)=\dim_{\Q}\bigl((G/H)\otimes_{\Z}\Q\bigr),
\]
with arbitrary cardinal values allowed.

\begin{lemma}[Rank--transcendence formula]\label{lem:rank-trdeg}
Let $k$ be a field, let $G$ be torsion-free abelian, and let $H\leq G$.
Then
\[
  \trdeg_{\Kk H}\Kk G=\rk(G/H).
\]
\end{lemma}

\begin{proof}
Put $W=(G/H)\otimes_{\Z}\Q$.  The images in $W$ of the elements of $G$
span $W$ over $\Q$.  Hence one may choose a subset $(g_i)_{i\in I}$ of
$G$ whose images
\[
  \bar g_i=(g_i+H)\otimes 1
\]
form a $\Q$-basis of $W$.  We prove separately the algebraic independence
and the algebraicity assertions.

First suppose that a finite polynomial relation over $\Kk H$ is given:
\begin{equation}\label{eq:polynomial-relation}
  \sum_{\nu\in S} a_\nu\prod_{i\in I}(\mon{g_i})^{\nu_i}=0.
\end{equation}
Here $S\subseteq\N^{(I)}$ is finite, each $\nu=(\nu_i)$ has finite support,
and $a_\nu\in\Kk H$.  Choose a common nonzero denominator $b\in k[H]$
for the finitely many coefficients $a_\nu$.  After multiplying
\cref{eq:polynomial-relation} by $b$, we obtain
\begin{equation}\label{eq:group-ring-relation}
  \sum_{\nu\in S} A_\nu\mon{\sum_i\nu_i g_i}=0
  \qquad(A_\nu\in k[H]).
\end{equation}
Expand $A_\nu=\sum_{h\in T_\nu}c_{\nu,h}\mon h$, where $T_\nu$ is finite.
If two monomials occurring in this expansion have the same exponent, then
for some $\nu,\mu\in S$ and $h\in T_\nu$, $h'\in T_\mu$,
\[
  h+\sum_i\nu_i g_i=h'+\sum_i\mu_i g_i.
\]
Passing to $W$ gives
\[
  \sum_i(\nu_i-\mu_i)\bar g_i=0.
\]
The vectors $\bar g_i$ are $\Q$-linearly independent, so $\nu_i=\mu_i$
for every $i$, and therefore $\nu=\mu$.  Thus terms belonging to distinct
multi-indices in \cref{eq:group-ring-relation} cannot cancel.  Since the
formal monomials $\mon g$, $g\in G$, are a $k$-basis of $k[G]$, every
$A_\nu$ is zero.  Consequently every $a_\nu$ is zero.  This proves that
the family $(\mon{g_i})_{i\in I}$ is algebraically independent over
$\Kk H$.

Now fix $g\in G$.  Because the $\bar g_i$ form a basis of $W$, there is a
finite subset $I_0\subseteq I$ and rational numbers $q_i$ such that
\[
  (g+H)\otimes1=\sum_{i\in I_0}q_i\bar g_i.
\]
Choose a positive integer $d$ with $m_i=dq_i\in\Z$ for all $i\in I_0$.
Then the element
\[
  x=dg-\sum_{i\in I_0}m_i g_i
\]
maps to zero in $(G/H)\otimes\Q$.  The kernel of the natural map
$G/H\to(G/H)\otimes\Q$ is the torsion subgroup of $G/H$: this follows
directly by viewing tensoring with $\Q$ as localization at the nonzero
integers.  Hence there is a
positive integer $r$ such that $rx\in H$.  Writing $h=rx$, we obtain the
actual equality in $G$
\begin{equation}\label{eq:clear-group-denominators}
  rdg=h+\sum_{i\in I_0}rm_i g_i.
\end{equation}
Let
\[
  L=\Kk H(\mon{g_i}:i\in I).
\]
Taking monomials in \cref{eq:clear-group-denominators} yields
\[
  (\mon g)^{rd}=\mon h\prod_{i\in I_0}(\mon{g_i})^{rm_i}\in L.
\]
Negative exponents cause no difficulty because $L$ is a field.  Thus
$\mon g$ is algebraic over $L$, being a root of
\[
  T^{rd}-\mon h\prod_{i\in I_0}(\mon{g_i})^{rm_i}\in L[T].
\]
Every element of $k[G]$ is a finite sum of elements algebraic over $L$, and
the elements algebraic over a field form a field.  It follows that
$k[G]$, and therefore its fraction field $\Kk G$, is algebraic over $L$.
The algebraically independent family $(\mon{g_i})_{i\in I}$ is consequently
a transcendence basis, and its cardinality is
$|I|=\dim_\Q W=\rk(G/H)$.
\end{proof}

\section{The One-Variable Puiseux Quotient}\label{sec:puiseux}

Let $F$ be a field of characteristic zero and let $t$ be transcendental over
$F$.  In a fixed algebraic closure of $F(t)$ choose elements
$t^{1/n!}$ compatibly, meaning that
\[
  \bigl(t^{1/(n+1)!}\bigr)^{n+1}=t^{1/n!}
  \quad(n\geq1).
\]
For $q=a/n!\in\Q$, set $t^q=(t^{1/n!})^a$; compatibility makes this
independent of the chosen factorial denominator: if $m\geq n$, then
$t^{1/n!}=(t^{1/m!})^{m!/n!}$.  We use the distinct
notations
\[
  t^\Q=\{t^q:q\in\Q\}\subseteq F(t^\Q)^\times,
  \qquad
  F(t^{\Q})=\bigcup_{n\geq1}F(t^{1/n!}).
\]
Thus $t^\Q$ is a multiplicative subgroup, whereas $F(t^\Q)$ is a field.
The following lemma is the local factorization engine of the paper.

\begin{lemma}[Puiseux quotient]\label{lem:puiseux-free}
If $\operatorname{char}F=0$, then
\[
  P(F,t):=F(t^{\Q})^\times/(F^\times t^{\Q})
\]
is a free abelian group.
\end{lemma}

\begin{proof}
Put $z_n=t^{1/n!}$ and
\[
  P_n=F(z_n)^\times/(F^\times z_n^{\Z}).
\]
Let $\mathcal I_n$ be the set of monic irreducible polynomials in the
polynomial ring $F[z_n]$, with the polynomial $z_n$ omitted.  For each
$r\in\mathcal I_n$, let $v_r$ be the usual exponent valuation in the unique
factorization of a rational function.  Every $f\in F(z_n)^\times$ has a
unique expression
\begin{equation}\label{eq:rational-factorization}
  f=c z_n^m\prod_{r\in\mathcal I_n}r(z_n)^{a_r},
\end{equation}
where $c\in F^\times$, $m,a_r\in\Z$, and only finitely many $a_r$ are
nonzero.  Therefore the map
\[
  [f]\longmapsto (v_r(f))_{r\in\mathcal I_n}
\]
is a well-defined isomorphism
\begin{equation}\label{eq:Pn-free}
  P_n\xrightarrow{\ \sim\ }\Z^{(\mathcal I_n)}.
\end{equation}

We next determine the transition map under these bases.  Put $d=n+1$ and
write $Y=z_{n+1}$, so that
\begin{equation}\label{eq:stage-substitution}
  z_n=Y^d.
\end{equation}
Fix a monic irreducible $p(Z)\in F[Z]$ with $p\neq Z$.  Then $p(0)\neq0$.
Because $F$ has characteristic zero, $p$ is separable and
$\gcd(p,p')=1$.  Choose $A,B\in F[Z]$ satisfying
\[
  A(Z)p(Z)+B(Z)p'(Z)=1.
\]
After substituting $Z=Y^d$, this becomes
\begin{equation}\label{eq:bezout-substitution}
  A(Y^d)p(Y^d)+B(Y^d)p'(Y^d)=1.
\end{equation}
The formal derivative of $p(Y^d)$ is
\[
  \frac{d}{dY}p(Y^d)=dY^{d-1}p'(Y^d).
\]
Equation \cref{eq:bezout-substitution} shows that $p(Y^d)$ is coprime to
$p'(Y^d)$; its nonzero constant term shows that it is coprime to $Y$; and
$d\neq0$ in $F$.  Hence $p(Y^d)$ is coprime to its derivative and is
square-free.

If $p_1$ and $p_2$ are distinct monic irreducibles, then they are coprime.
Substitution into a Bezout identity for $p_1$ and $p_2$ shows that
$p_1(Y^d)$ and $p_2(Y^d)$ are coprime.  Thus no target irreducible can
occur above two distinct source irreducibles.

Conversely, let $q(Y)\neq Y$ be a monic irreducible polynomial over $F$,
and write $E=F[Y]/(q)$ and $\bar y=Y+(q)\in E$.  Since $q\neq Y$, we have
$\bar y\neq0$.  Let $p(Z)$ be the monic minimal polynomial of $\bar y^d$
over $F$.  Then $p$ is irreducible, $p\neq Z$, and
$p(\bar y^d)=0$.  The last equality is precisely the assertion that
$q(Y)$ divides $p(Y^d)$.  If $q$ also divides $\widetilde p(Y^d)$ for a
monic irreducible $\widetilde p$, then
$\widetilde p(\bar y^d)=0$, so minimality implies
$p=\widetilde p$.  Therefore each target irreducible lies above exactly one
source irreducible.

For $p\in\mathcal I_n$, let
\[
  S_p=\{q\in\mathcal I_{n+1}:q(Y)\text{ divides }p(Y^d)\}.
\]
Each $S_p$ is a nonempty finite set, the sets $S_p$ are pairwise disjoint,
and the preceding paragraph proves that they partition
$\mathcal I_{n+1}$.  Square-freeness gives
\[
  p(Y^d)=\prod_{q\in S_p}q(Y).
\]
Consequently, under \cref{eq:Pn-free}, the transition homomorphism sends
the basis vector $p$ to
\begin{equation}\label{eq:transition-vector}
  \sum_{q\in S_p}q.
\end{equation}
For each $p$, choose one element $q_p\in S_p$.  The family
\[
  \left\{\sum_{q\in S_p}q\right\}
  \cup\{q:q\in S_p\setminus\{q_p\}\}
\]
is a basis of $\Z^{(S_p)}$: relative to the original basis, replacing
$q_p$ by the sum has determinant $1$.  Thus
\[
  \Z^{(S_p)}=
  \Z\left(\sum_{q\in S_p}q\right)
  \oplus\Z^{(S_p\setminus\{q_p\})}.
\]
Taking the direct sum over all $p$ proves that
$P_n\to P_{n+1}$ is injective, that its image is a direct summand, and that
its cokernel is free abelian.

It remains to justify that these stage groups really form an increasing
union inside the desired quotient.  Let
\[
  \iota_n:P_n\longrightarrow
  P(F,t)=F(t^\Q)^\times/(F^\times t^\Q)
\]
be induced by the field inclusion.  Suppose $[f]\in\ker\iota_n$.  Then
$f=c t^q$ for some $c\in F^\times$ and $q\in\Q$.  Write
$f=a/b$ with nonzero $a,b\in F[z_n]$.  In the group algebra $F[\Q]$ we
then have
\begin{equation}\label{eq:support-translation-puiseux}
  a=c t^q b.
\end{equation}
For an element of the group algebra written uniquely as a finite sum
\[
  u=\sum_{r\in\Q}u_r t^r\in F[\Q],
\]
define its support by
\[
  \operatorname{supp}(u)=\{r\in\Q:u_r\neq0\},
  \qquad \operatorname{supp}(0)=\varnothing.
\]
This definition is used only for group-algebra elements, not for arbitrary
fractions.  With $F[z_n]$ viewed as a subring of $F[\Q]$, both $a$ and $b$
have finite support contained in
$L_n=(1/n!)\Z$.  Multiplication by $t^q$ translates support and causes no
cancellation, so
\[
  \operatorname{supp}(a)=q+\operatorname{supp}(b).
\]
Choose $s\in\operatorname{supp}(b)$.  Then $s\in L_n$ and
$q+s\in L_n$, whence $q\in L_n$.  Therefore $t^q\in z_n^\Z$, so $[f]$ was
already the identity element of $P_n$.  This proves that every $\iota_n$
is injective.

Every rational exponent belongs to some $L_n$, and every element of the
union field belongs to some $F(z_n)$.  Hence the images of the $P_n$ cover
$P(F,t)$.  We may therefore identify
\[
  P(F,t)=\bigcup_{n\geq1}P_n.
\]
For each $n$, choose the explicit free complement $C_n$ constructed above,
so that $P_{n+1}=P_n\oplus C_n$.  Induction gives
\[
  P_m=P_1\oplus\bigoplus_{1\leq n<m}C_n,
\]
and taking the union over $m$ gives
\[
  P(F,t)=P_1\oplus\bigoplus_{n\geq1}C_n.
\]
Both $P_1$ and all the $C_n$ are free abelian; their direct sum is free
abelian.  This proves the lemma.
\end{proof}

\section{Freeness of the Monomial Defect}\label{sec:defect}

We now pass from one divisible direction to an arbitrary torsion-free
abelian group.

\begin{definition}
For a field $k$ and a torsion-free abelian group $G$, define
\[
  \Dk G=\Kk G^\times/(k^\times X^G).
\]
We use multiplicative notation for this quotient; its identity element is
the class $[1]$.  When discussing freeness and internal direct-sum
decompositions, we may regard it as an abstract abelian group.
\end{definition}

We shall use the following elementary module-theoretic fact.  Its
arbitrary-rank form is included because finite-rank Smith normal form is not
sufficient here.

\begin{lemma}[Subgroups of free abelian groups]\label{lem:subgroup-free}
Every subgroup of a free abelian group is free abelian.
\end{lemma}

\begin{proof}
Let $A=\bigoplus_{\alpha<\kappa}\Z e_\alpha$ and let $B\leq A$.  For each
$\alpha\leq\kappa$, define
\[
  A_\alpha=\bigoplus_{\beta<\alpha}\Z e_\beta,
  \qquad B_\alpha=B\cap A_\alpha.
\]
In particular, $A_0=0$ and $B_0=B\cap A_0=0$.
The chain $(B_\alpha)$ is increasing and continuous: if $\lambda$ is a
limit ordinal, every element of $A_\lambda$ has finite support, so it
belongs to some $A_\alpha$ with $\alpha<\lambda$, and hence
$B_\lambda=\bigcup_{\alpha<\lambda}B_\alpha$.

At a successor stage, projection onto the coefficient of $e_\alpha$
induces an injective homomorphism
\[
  B_{\alpha+1}/B_\alpha\longrightarrow\Z.
\]
Its image is a subgroup of $\Z$, therefore either $0$ or $d_\alpha\Z$ for
some $d_\alpha>0$.  Thus $B_{\alpha+1}/B_\alpha$ is either zero or infinite
cyclic.  In either case it is free and hence projective, so the exact
sequence
\[
  0\longrightarrow B_\alpha\longrightarrow B_{\alpha+1}
  \longrightarrow B_{\alpha+1}/B_\alpha\longrightarrow0
\]
splits.  Choose a complement $C_\alpha$, equal either to $0$ or to a cyclic
group, such that $B_{\alpha+1}=B_\alpha\oplus C_\alpha$.  Transfinite
induction, using continuity at limit ordinals, gives
\[
  B_\alpha=\bigoplus_{\beta<\alpha}C_\beta
  \quad\text{for every }\alpha\leq\kappa.
\]
In particular $B=\bigoplus_{\beta<\kappa}C_\beta$ is free abelian.
\end{proof}

\begin{lemma}[Passage to an overgroup]\label{lem:defect-injective}
If $G\leq V$ are torsion-free abelian groups, then the natural map
\[
  \Dk G\longrightarrow\Dk V
\]
is injective.
\end{lemma}

\begin{proof}
The inclusion $k[G]\hookrightarrow k[V]$ is an inclusion of integral
domains and therefore extends to an injective field homomorphism
$\Kk G\hookrightarrow\Kk V$.  It sends $k^\times X^G$ into
$k^\times X^V$, so the displayed homomorphism of quotient groups is
well-defined.

Suppose that $f\in\Kk G^\times$ maps to $c\mon v$ for some
$c\in k^\times$ and $v\in V$.  Write $f=a/b$ with nonzero
$a,b\in k[G]$.  In $k[V]$ we have
\[
  a=c\mon v b.
\]
Multiplication by a monomial translates support without cancellation, so
\[
  \operatorname{supp}(a)=v+\operatorname{supp}(b).
\]
Choose $g\in\operatorname{supp}(b)$, which is nonempty because $b\neq0$.
We have $g\in G$, and the support equality gives $v+g\in
\operatorname{supp}(a)\subseteq G$.  Therefore
$v=(v+g)-g\in G$.  The equality $f=c\mon v$ now holds already in
$\Kk G$, so $f\in k^\times X^G$.  The kernel is trivial.
\end{proof}

\begin{proposition}[Divisible case]\label{prop:defect-vector-space}
Let $k$ have characteristic zero and let $V$ be a $\Q$-vector space,
viewed as an additive group.  Then $\Dk V$ is free abelian.
\end{proposition}

\begin{proof}
Choose a well-ordered basis $(e_\alpha)_{\alpha<\kappa}$ of $V$ and put
\[
  V_\alpha=\bigoplus_{\beta<\alpha}\Q e_\beta.
\]
Write $D_\alpha=\Dk{V_\alpha}$.  By
\cref{lem:defect-injective}, all maps $D_\alpha\to D_\beta$ for
$\alpha\leq\beta$ are injective, and we henceforth identify $D_\alpha$
with its image in $D_\beta$.

At a successor stage, put $F_\alpha=\Kk{V_\alpha}$ and
$t_\alpha=\mon{e_\alpha}$.  Since
$V_{\alpha+1}=V_\alpha\oplus\Q e_\alpha$, the group algebra is generated
by $k[V_\alpha]$ and the monomials $t_\alpha^q$, $q\in\Q$, and therefore
\[
  \Kk{V_{\alpha+1}}
  =F_\alpha(t_\alpha^{\Q}).
\]
The element $t_\alpha$ is transcendental over $F_\alpha$.  Indeed, a
polynomial relation with coefficients in $F_\alpha$ could first be
multiplied by a common nonzero denominator from $k[V_\alpha]$.  Expanding
the resulting relation in $k[V_{\alpha+1}]$ would give a nontrivial
$k$-linear relation among monomials whose $e_\alpha$-coordinates are
distinct, which is impossible.

We now compute the successor quotient without suppressing any subgroup.
Under the quotient map
\[
  \Kk{V_{\alpha+1}}^\times\longrightarrow D_{\alpha+1},
\]
the inverse image of $D_\alpha$ is
\[
  F_\alpha^\times(k^\times X^{V_{\alpha+1}})
  =F_\alpha^\times t_\alpha^\Q,
\]
because $X^{V_{\alpha+1}}=X^{V_\alpha}t_\alpha^\Q$ and
$k^\times X^{V_\alpha}\subseteq F_\alpha^\times$.  The third isomorphism
theorem therefore gives
\begin{align*}
  D_{\alpha+1}/D_\alpha
  &\cong
  \Kk{V_{\alpha+1}}^\times/
  \bigl(F_\alpha^\times t_\alpha^{\Q}\bigr)\\
  &=F_\alpha(t_\alpha^\Q)^\times/
  \bigl(F_\alpha^\times t_\alpha^\Q\bigr)\\
  &=P(F_\alpha,t_\alpha).
\end{align*}
This quotient is free by \cref{lem:puiseux-free}, so the short exact
sequence
\[
  0\longrightarrow D_\alpha\longrightarrow D_{\alpha+1}
  \longrightarrow P(F_\alpha,t_\alpha)\longrightarrow0
\]
splits: a free abelian quotient is projective.  Choose a free subgroup
$C_\alpha\leq D_{\alpha+1}$ such that
\begin{equation}\label{eq:defect-successor-splitting}
  D_{\alpha+1}=D_\alpha\oplus C_\alpha.
\end{equation}

At a limit ordinal $\lambda$, let $[f]\in D_\lambda$.  Represent $f$ as
$a/b$ with nonzero $a,b\in k[V_\lambda]$.  The two supports are finite, so
only finitely many basis vectors $e_\beta$ occur in all their exponents.
Because $\lambda$ is a limit ordinal, there is some $\alpha<\lambda$
larger than every such index.  Then $a,b\in k[V_\alpha]$ and
$f\in\Kk{V_\alpha}$, which shows that $[f]\in D_\alpha$.  The reverse
inclusion is automatic.  Hence the chain is continuous:
\[
  D_\lambda=\bigcup_{\alpha<\lambda}D_\alpha.
\]
Starting with $D_0=0$, transfinite induction using
\cref{eq:defect-successor-splitting} at successors and continuity at limits
gives
\[
  D_\alpha=\bigoplus_{\beta<\alpha}C_\beta
  \quad(\alpha\leq\kappa).
\]
Since $V_\kappa=V$, we obtain
$\Dk V=\bigoplus_{\beta<\kappa}C_\beta$, a direct sum of free abelian
groups and therefore a free abelian group.
\end{proof}

\begin{theorem}[Defect freeness]\label{thm:defect-free}
Let $k$ be a field of characteristic zero and let $G$ be any torsion-free
abelian group.  Then $\Dk G$ is free abelian.
\end{theorem}

\begin{proof}
Embed $G$ in its divisible hull $V=\Q\otimes_{\Z}G$.  The canonical map
$G\to V$, $g\mapsto g\otimes1$, is injective because its
kernel is the torsion subgroup of $G$, which is zero.  By
\cref{lem:defect-injective}, $\Dk G$ is a subgroup of $\Dk V$, which is free
by \cref{prop:defect-vector-space}.  Apply \cref{lem:subgroup-free} to this
subgroup.
\end{proof}

\begin{remark}[Why characteristic zero matters]\label{rem:char-p}
The mechanism in \cref{lem:puiseux-free} fails in characteristic $p>0$.
Indeed,
\[
  t-1=(t^{1/p^n}-1)^{p^n}.
\]
Put
\[
  P=F(t^{\Q})^\times/(F^\times t^{\Q}),
  \qquad u_n=t^{1/p^n},
  \qquad y_n=[u_n-1]\in P.
\]
We retain multiplicative notation in $P$.  Since
$u_n-1=(u_{n+1}-1)^p$, we have $y_n=y_{n+1}^p$.  Define
\[
  \eta:\Z[1/p]\longrightarrow
  P,
  \qquad \eta(a/p^n)=y_n^a
  \quad(a\in\Z,\ n\geq0).
\]
Here $\Z[1/p]$ is written additively and $P$ multiplicatively.  To check
well-definedness, suppose $a/p^n=b/p^m$ and choose
$\ell\geq m,n$.  Then $ap^{\ell-n}=bp^{\ell-m}$ and
\[
  y_n^a=y_\ell^{ap^{\ell-n}}
  =y_\ell^{bp^{\ell-m}}=y_m^b.
\]
The same common-denominator calculation gives
$\eta(r+s)=\eta(r)\eta(s)$, so $\eta$ is a group homomorphism from the
additively written source to the multiplicatively written target.
The map is injective.  Indeed, suppose $y_n^a=1$ in $P$.  By taking
inverses if necessary, it is enough to consider $a>0$.  The definition of
the quotient then gives $c\in F^\times$ and $q\in\Q$ such that
\[
  (u_n-1)^a=ct^q.
\]
Choose a positive integer $N$ divisible by $p^n$ and by a denominator of
$q$, and put $w=t^{1/N}$.  The equality lies in $F(w)$ and becomes
\[
  \bigl(w^{N/p^n}-1\bigr)^a=cw^{Nq}.
\]
At $w=1$ the left-hand side is zero, whereas the right-hand side is
$c\neq0$, a contradiction.  Thus $a=0$.  The image of $\eta$ is exactly
the subgroup generated by the $y_n$, so this subgroup is isomorphic to
$\Z[1/p]$.  Since a subgroup of a free abelian group must be free by
\cref{lem:subgroup-free}, while $\Z[1/p]$ is not free abelian (its element
$1$ is divisible by $p^n$ for every $n$, whereas no nonzero element of a
free abelian group has this property), the Puiseux
quotient is not free in characteristic
$p$.  This observation marks a genuine boundary of the present proof; it
does not decide the positive-characteristic reconstruction problem.
\end{remark}

\section{Reconstruction from Two Monomial Frames}\label{sec:reconstruction}

We now prove \cref{thm:intro-main}.  The argument is stated separately to
make clear that defect freeness is the only structural input about
multiplicative groups.

\begin{theorem}[Two-frame reconstruction]\label{thm:reconstruction}
Let $G$ and $H$ be torsion-free abelian groups.  Every field isomorphism
\[
  \varphi:\KQ G\longrightarrow\KQ H
\]
implies an isomorphism $G\cong H$.
\end{theorem}

\begin{proof}
Every field isomorphism in characteristic zero fixes the prime field $\Q$.
Let
\[
  \Gamma_G=\Q^\times X^G,
  \qquad
  \Gamma_H=\Q^\times X^H.
\]
Define
\begin{align*}
  M&=\{g\in G:\varphi(\mon g)\in\Gamma_H\},\\
  N&=\{h\in H:\varphi^{-1}(\mon h)\in\Gamma_G\}.
\end{align*}
They are subgroups: for example, if $g_1,g_2\in M$, then
\[
  \varphi(\mon{g_1-g_2})
  =\varphi(\mon{g_1})\varphi(\mon{g_2})^{-1}\in\Gamma_H,
\]
and the verification for $N$ is identical.  For $g\in M$, there are unique
$c_g\in\Q^\times$ and $\theta(g)\in H$ such that
\[
  \varphi(\mon g)=c_g\mon{\theta(g)}.
\]
Uniqueness holds because
$c\mon h=c'\mon{h'}$ in $\Q[H]$ forces $h=h'$ by linear independence of
the standard monomial basis, and then $c=c'$.  Multiplying the formulas for
$g_1$ and $g_2$ shows that
\[
  c_{g_1+g_2}=c_{g_1}c_{g_2},
  \qquad
  \theta(g_1+g_2)=\theta(g_1)+\theta(g_2),
\]
so the exponent map $\theta:M\to H$ is a homomorphism.  It is injective: if
$\theta(g)=0$, then $\varphi(\mon g)\in\Q^\times$, and applying
$\varphi^{-1}$, which fixes $\Q$, shows that $\mon g\in\Q^\times$.  The
monomial basis of $\Q[G]$ then gives $g=0$.

We next prove $\theta(M)=N$ explicitly.  If $g\in M$ and
$h=\theta(g)$, then
\[
  \varphi^{-1}(\mon h)=c_g^{-1}\mon g\in\Gamma_G,
\]
so $h\in N$.  Conversely, if $h\in N$, write
$\varphi^{-1}(\mon h)=d\mon g$ with $d\in\Q^\times$ and $g\in G$.
Applying $\varphi$ gives
$\varphi(\mon g)=d^{-1}\mon h$; hence $g\in M$ and $\theta(g)=h$.
Therefore
\begin{equation}\label{eq:M-N-iso}
  \theta:M\xrightarrow{\ \sim\ }N.
\end{equation}

Consider the group homomorphism from the additive group $G$ to the
multiplicative group $\DQ H$,
\[
  \delta_\varphi:G\longrightarrow\DQ H,
  \qquad g\longmapsto[\varphi(\mon g)].
\]
It is a homomorphism because
\[
  \delta_\varphi(g_1+g_2)
  =[\varphi(\mon{g_1})\varphi(\mon{g_2})]
  =\delta_\varphi(g_1)\delta_\varphi(g_2).
\]
By the definition of the defect quotient,
\[
  \delta_\varphi(g)=1
  \quad\Longleftrightarrow\quad
  \varphi(\mon g)\in\Q^\times X^H
  \quad\Longleftrightarrow\quad g\in M.
\]
The first isomorphism theorem therefore embeds $G/M$ into $\DQ H$.
The latter is free by \cref{thm:defect-free}; hence $G/M$ is free by
\cref{lem:subgroup-free}.  Applying the same argument to $\varphi^{-1}$
embeds $H/N$ into $\DQ G$, and so $H/N$ is free as well.

Choose bases of these free quotients, indexed by sets $I$ and $J$:
\[
  G/M\cong\Z^{(I)},\qquad H/N\cong\Z^{(J)}.
\]
A free abelian group is projective: choosing arbitrary lifts of its basis
elements defines a section of any surjection onto it.  Thus the quotient
maps $G\to G/M$ and $H\to H/N$ split, and we obtain
\begin{equation}\label{eq:splittings}
  G\cong M\oplus\Z^{(I)},
  \qquad
  H\cong N\oplus\Z^{(J)}.
\end{equation}

The restriction of $\varphi$ identifies the two group subfields:
\begin{equation}\label{eq:subfields}
  \varphi(\KQ M)=\KQ N.
\end{equation}
Indeed, for every $g\in M$ the formula
$\varphi(\mon g)=c_g\mon{\theta(g)}$ shows that
$\varphi(\KQ M)\subseteq\KQ N$.  Conversely, if $h\in N$, the surjectivity
of $\theta$ supplies $g\in M$ with $h=\theta(g)$, and then
$\mon h=c_g^{-1}\varphi(\mon g)$ belongs to $\varphi(\KQ M)$.  Since the
monomials $X^h$, $h\in N$, together with $\Q$ generate $\KQ N$ as a field,
the reverse inclusion follows.

Equation \cref{eq:subfields} means that $\varphi$ is an isomorphism of the
field extensions
\[
  \KQ G/\KQ M\quad\text{and}\quad\KQ H/\KQ N.
\]
An isomorphism of field extensions takes an algebraically independent set
over the first base field to an algebraically independent set over the
second, and the same statement for $\varphi^{-1}$ gives equality of the
relative transcendence degrees.  Using \cref{lem:rank-trdeg} and
\cref{eq:splittings}, we obtain
\begin{align*}
  |I|
  &=\rk(G/M)\\
  &=\trdeg_{\KQ M}\KQ G\\
  &=\trdeg_{\KQ N}\KQ H\\
  &=\rk(H/N)\\
  &=|J|.
\end{align*}
Hence $\Z^{(I)}\cong\Z^{(J)}$.  Combining this with
\cref{eq:M-N-iso,eq:splittings} gives the explicit chain
\[
  G\cong M\oplus\Z^{(I)}
  \cong N\oplus\Z^{(J)}
  \cong H.
\]
\end{proof}

\begin{proof}[Proof of \cref{thm:intro-main}]
The forward implication is \cref{thm:reconstruction}.  A group isomorphism
$G\to H$ extends $\Q$-linearly to an isomorphism of group algebras and then
to an isomorphism of their fraction fields, proving the converse.
\end{proof}

\section{Stabilization and Rickard's Example}\label{sec:applications}

The reconstruction theorem does more than distinguish individual groups: it
recovers every free-stabilization relation.

\begin{corollary}[Exact stabilization profile]\label{cor:stabilization}
Let $G,H$ be torsion-free abelian groups and let $I,J$ be arbitrary sets.
For algebraically independent families $(t_i)_{i\in I}$ and
$(u_j)_{j\in J}$,
\[
  \KQ G(t_i:i\in I)\cong\KQ H(u_j:j\in J)
\]
if and only if
\[
  G\oplus\Z^{(I)}\cong H\oplus\Z^{(J)}.
\]
\end{corollary}

\begin{proof}
There are natural isomorphisms
\[
  \KQ{(G\oplus\Z^{(I)})}\cong\KQ G(t_i:i\in I)
\]
and similarly for $H$ and $J$.  Apply \cref{thm:intro-main}.
\end{proof}

We finish with the motivating application.  Let
\[
  \mathcal O=\Z[\sqrt2],
  \qquad
  A=\left\{(a_n)_{n\geq0}\in\mathcal O^{\N}:
  \sup_n|a_n|<\infty\right\},
\]
where $\mathcal O$ is embedded in $\mathbb R$ in the usual way.  Rickard
proved that this torsion-free abelian group satisfies
\[
  A\cong A\oplus\Z^2,
  \qquad
  A\not\cong A\oplus\Z
\]
\cite{Rickard2020}.

\begin{corollary}[Dimension-two field stabilization]\label{cor:field-example}
For $F=\KQ A$ and algebraically independent $x,y$ over $F$,
\[
  F\cong F(x,y),
  \qquad
  F\not\cong F(x).
\]
\end{corollary}

\begin{proof}
The isomorphism $A\cong A\oplus\Z^2$ and
\cref{cor:stabilization} give $F\cong F(x,y)$.  If $F\cong F(x)$, then
\cref{thm:intro-main} applied to
\[
  \KQ A\cong F(x)\cong\KQ{(A\oplus\Z)}
\]
would imply $A\cong A\oplus\Z$, contrary to Rickard's theorem.
\end{proof}

\section{Scope and Further Questions}

The argument reconstructs the abstract isomorphism type of $G$ from the
abstract field $\KQ G$.  It does not assert that the standard monomial group
$\Q^\times X^G$ is definable without parameters, nor does it construct a
canonical copy of $G$ inside the field.  What is canonical in the proof is
the comparison of two monomial frames after an isomorphism has been given.

The torsion-free hypothesis is also intrinsic to the construction: if $G$
has nontrivial torsion, the rational group algebra generally has zero
divisors, so its ordinary fraction field is unavailable.  Finally,
\cref{rem:char-p} shows that the central freeness theorem fails in positive
characteristic.  Determining whether the corresponding reconstruction
statement nevertheless holds over finite fields requires a different
invariant.

\section*{Acknowledgements}

The authors thank Jeremy Rickard for formulating the rational group-field
isomorphism question and for the bounded-sequence group used in
\cref{cor:field-example}.  Warren May's work supplied the historical context
for the multiplicative-group freeness phenomenon.

The discovery and structural organization of the proofs and the navigation of
the literature were completed entirely by OpenAI's ChatGPT and Codex.  LaTeX
typesetting and text revision were carried out jointly by ChatGPT, Codex, and
the authors.  All mathematical arguments were independently verified by the
authors.  ChatGPT and Codex are neither authors nor referees.


\begingroup
\footnotesize
\begin{thebibliography}{99}
\setlength{\itemsep}{0.2em}

\bibitem[BV88]{BastosViswanathan1988}
G.~G. Bastos and T.~M. Viswanathan,
\newblock Torsion-free abelian groups, valuations and twisted group rings,
\newblock \emph{Canadian Mathematical Bulletin} \textbf{31} (1988), no.~2,
139--146,
\newblock
\href{https://doi.org/10.4153/CMB-1988-021-x}{doi:10.4153/CMB-1988-021-x}.

\bibitem[May72]{May1972}
W.~May,
\newblock Multiplicative groups of fields,
\newblock \emph{Proceedings of the London Mathematical Society} (3)
\textbf{24} (1972), 295--306,
\newblock
\href{https://doi.org/10.1112/plms/s3-24.2.295}{doi:10.1112/plms/s3-24.2.295}.

\bibitem[May79]{May1979}
W.~May,
\newblock Multiplicative groups under field extension,
\newblock \emph{Canadian Journal of Mathematics} \textbf{31} (1979), no.~2,
436--440,
\newblock
\href{https://doi.org/10.4153/CJM-1979-047-5}{doi:10.4153/CJM-1979-047-5}.

\bibitem[Ric17]{RickardMO2017}
J.~Rickard,
\newblock The field of fractions of the rational group algebra of a
torsion-free abelian group,
\newblock MathOverflow question 259117 (2017),
\newblock
\href{https://mathoverflow.net/questions/259117/}{mathoverflow.net/questions/259117}.

\bibitem[Ric20]{Rickard2020}
J.~Rickard,
\newblock Pathological abelian groups: A friendly example,
\newblock \emph{Journal of Algebra} \textbf{558} (2020), 640--645,
\newblock
\href{https://doi.org/10.1016/j.jalgebra.2019.10.014}{doi:10.1016/j.jalgebra.2019.10.014};
\newblock
\href{https://arxiv.org/abs/1904.09327}{arXiv:1904.09327}.

\end{thebibliography}
\endgroup
\end{document}